\documentclass[a4paper]{article}
\usepackage[utf8]{inputenc}
\usepackage{a4wide}
\usepackage{amsfonts,amsmath,amsthm,amssymb}
\usepackage{enumitem}
\usepackage{tikz-cd}
\usepackage{enumitem}  
\usepackage{caption}
\usepackage{subcaption} 

\usepackage{tikz}
\usetikzlibrary{arrows.meta, positioning, math, shapes, backgrounds}
\usepackage{stmaryrd} 

\newcommand{\figurezero}
{
\begin{tikzpicture}[baseline={([yshift=-1ex]current bounding box.center)}, scale=1.5, vertex/.style={draw,circle,scale=0.5,thick,fill=black!100, /pgf/outer sep=5},  arc/.style= {->,thick, >={Triangle[length=0.6em, scale width=0.8]}}]
	\foreach [count=\ile] \i/\j in {0/v, 72/u, 144/w', 216/v', 288/w}{
		\node[vertex] (p\ile) at (\i : 1){};
		\node (q\ile) at (\i : 1.2) {$\j$};
	};
	
    \begin{scope}[on background layer]
	\foreach \i/\j in {4/1, 3/5, 5/4, 1/3, 2/1, 2/4, 5/2, 3/2}{
		\draw[arc, color=black!20] (p\i) edge (p\j);
	};
	
	\foreach \i/\j in {1/5, 4/3}{
		\draw[arc, color=black!10!green] (p\i) edge (p\j);
	};
	\end{scope}
\end{tikzpicture}
}

\newcommand{\figureTa}{
\begin{tikzpicture}[baseline={([yshift=-1ex]current bounding box.center)}, scale=1, vertex/.style={draw,circle,scale=0.5,thick,fill=black!100, /pgf/outer sep=5},  arc/.style= {->,thick, >={Triangle[length=0.6em, scale width=0.8]}}]

	\foreach[count=\j] \i/\k in {18/w_1, 90/w_2, 162/w_3, 234/w_4, 306/w_5}{
		\node[vertex] (p\j) at (\i : 1){};
		\node (q\j) at (\i : 1.3){$\k$};
	};
	
	\foreach \i/\j in {1/2, 2/3, 3/4, 4/5, 5/1, 3/1, 4/2, 5/3, 4/1, 2/5}{
		\draw[arc] (p\i) edge (p\j);
	};
\end{tikzpicture}
}

\newcommand{\figureTb}{
\begin{tikzpicture}[baseline={([yshift=-1ex]current bounding box.center)}, scale=1, vertex/.style={draw,circle,scale=0.5,thick,fill=black!100, /pgf/outer sep=5},  arc/.style= {->,thick, >={Triangle[length=0.6em, scale width=0.8]}}]

	\foreach[count=\j] \i/\k in {18/w_1, 90/w_2, 162/w_3, 234/w_4, 306/w_5}{
		\node[vertex] (p\j) at (\i : 1){};
		\node (q\j) at (\i : 1.3){$\k$};
	};
	
	\foreach \i/\j in {5/2, 2/4, 1/4, 1/3, 3/5, 4/5, 1/2, 3/4, 5/1, 2/3}{
		\draw[arc] (p\i) edge (p\j);
	};
\end{tikzpicture}
}

\newcommand{\figureTc}{
\begin{tikzpicture}[baseline={([yshift=-1ex]current bounding box.center)}, scale=1, vertex/.style={draw,circle,scale=0.5,thick,fill=black!100, /pgf/outer sep=5},  arc/.style= {->,thick, >={Triangle[length=0.6em, scale width=0.8]}}]

	\foreach[count=\j] \i/\k in {18/w_1, 90/w_2, 162/w_3, 234/w_4, 306/w_5}{
		\node[vertex] (p\j) at (\i : 1){};
		\node (q\j) at (\i : 1.3){$\k$};
	};
	
	\foreach \i/\j in {1/2, 1/3, 1/4, 2/3, 2/4, 2/5, 3/4, 3/5, 4/5, 5/1}{
		\draw[arc] (p\i) edge (p\j);
	};
\end{tikzpicture}
}

\newcommand{\figureTd}{
\begin{tikzpicture}[baseline={([yshift=-1ex]current bounding box.center)}, scale=1, vertex/.style={draw,circle,scale=0.5,thick,fill=black!100, /pgf/outer sep=5},  arc/.style= {->,thick, >={Triangle[length=0.6em, scale width=0.8]}}]

	\foreach[count=\j] \i/\k in {18/w_1, 90/w_2, 162/w_3, 234/w_4, 306/w_5}{
		\node[vertex] (p\j) at (\i : 1){};
		\node (q\j) at (\i : 1.3){$\k$};
	};
	
	\foreach \i/\j in {1/2, 3/4, 4/5, 5/3, 1/3, 1/4, 5/1, 2/3, 2/4, 2/5}{
		\draw[arc] (p\i) edge (p\j);
	};
\end{tikzpicture}
}

\newcommand{\figureTe}{
\begin{tikzpicture}[baseline={([yshift=-1ex]current bounding box.center)}, scale=1, vertex/.style={draw,circle,scale=0.5,thick,fill=black!100, /pgf/outer sep=5},  arc/.style= {->,thick, >={Triangle[length=0.6em, scale width=0.8]}}]

	\foreach[count=\j] \i/\k in {18/w_1, 90/w_2, 162/w_3, 234/w_4, 306/w_5}{
		\node[vertex] (p\j) at (\i : 1){};
		\node (q\j) at (\i : 1.3){$\k$};
	};
	
	\foreach \i/\j in {1/2, 1/3, 1/4, 2/3, 4/2, 2/5, 3/4, 3/5, 4/5, 5/1}{
		\draw[arc] (p\i) edge (p\j);
	};
\end{tikzpicture}
}

\newcommand{\figureTlaygadget}{
\begin{tikzpicture}[baseline={([yshift=-1ex]current bounding box.center)}, scale=1, vertex/.style={draw,circle,scale=0.5,thick,fill=black!100, /pgf/outer sep=5},  arc/.style= {->,thick, >={Triangle[length=0.6em, scale width=0.8]}}]

	\foreach[count=\j] \i in {(0,0), (2,0), (1,1), (1,-1), (0,2), (2,2)}{
		\node[vertex] (p\j) at \i {};
	};
	
	\node (q1)[below=2pt] at (0,0){$u$};
	\node (q11)[left=2pt] at (0,0){$(0,1)$};
	\node (q2)[below=2pt] at (2,0){$v$};
	\node (q21)[right=2pt] at (2,0){$(1,0)$};
	\node (q3)[below=2pt] at (1,1){$w$};
	\node (q4)[above=2pt] at (1,-1){$z$};
	\node (q5)[above=2pt] at (0,2){$x$};
	\node (q6)[above=2pt] at (2,2){$y$};
	\node (q41)[below=2pt] at (1,-1){$(1,1)$};
	\node (q51)[left=2pt] at (0,2){$(1,2)$};
	\node (q61)[right=2pt] at (2,2){$(2,1)$};
		
	\foreach \i/\j in {1/2, 2/3, 3/1, 1/4, 4/2, 1/5, 5/6, 6/2}{
		\draw[arc] (p\i) edge (p\j);
	};
\end{tikzpicture}
}

\newcommand{\figureTlaygadgettwo}{
\begin{tikzpicture}[baseline={([yshift=-1ex]current bounding box.center)}, scale=1, vertex/.style={draw,circle,scale=0.5,thick,fill=black!100, /pgf/outer sep=5},  arc/.style= {->,thick, >={Triangle[length=0.6em, scale width=0.8]}}]

	\foreach[count=\j] \i in {(0,0), (2,0), (1,1), (1,-1), (0,2), (2,2)}{
		\node[vertex] (p\j) at \i {};
	};
	
	\node (q1)[below=2pt] at (0,0){$u$};
	\node (q11)[left=2pt] at (0,0){$(3,1)$};
	\node (q2)[below=2pt] at (2,0){$v$};
	\node (q21)[right=2pt] at (2,0){$(1,0) (1, 3)$};
	\node (q3)[below=2pt] at (1,1){$w$};
	\node (q4)[above=2pt] at (1,-1){$z$};
	\node (q5)[above=2pt] at (0,2){$x$};
	\node (q6)[above=2pt] at (2,2){$y$};
	\node (q41)[below=2pt] at (1,-1){$(0,1) (1,1)$};
	\node (q51)[left=2pt] at (0,2){$(1,2)$};
	\node (q61)[right=2pt] at (2,2){$(2,1)$};
		
	\foreach \i/\j in {1/2, 2/3, 3/1, 1/4, 4/2, 1/5, 5/6, 6/2}{
		\draw[arc] (p\i) edge (p\j);
	};
\end{tikzpicture}
}

\newcommand{\figureTg}{
\begin{tikzpicture}
[baseline={([yshift=-1ex]current bounding box.center)}, scale=1, vertex/.style={draw,circle,scale=0.5,thick,fill=black!100, /pgf/outer sep=5},  arc/.style= {->,thick, >={Triangle[length=0.6em, scale width=0.8]}}]

\foreach[count=\j] \i in {1,2,3,4,5,6}{
	\node[vertex] (p\i) at (\i,0){};
	\node (q\i) at (\i, -0.3){$v_{\i}$};
};

\node[vertex] (p0) at (3.5, 2){};
\node (q0) at (3.5, 2.3){$v_7$};

\foreach \i/\j in {1/2, 2/3, 3/4, 4/5, 5/6, 0/4, 3/0, 0/2, 5/0}{
	\draw[arc] (p\i) edge (p\j);
};

\foreach \i/\j in {0/1, 6/0}{
	\draw[arc, bend right] (p\i) edge (p\j);
};
\end{tikzpicture}
}

\usepackage{graphicx}
\newcommand{\abs}[1]{\left\vert#1\right\vert}
\newcommand{\CC}[1]{\overrightarrow{C\hspace{4pt}}_{\hspace{-5pt}#1}}
\newcommand{\TT}[1]{\overrightarrow{T\hspace{4pt}}_{\hspace{-5pt}#1}}
\newcommand{\PP}[1]{\overrightarrow{P\hspace{4pt}}_{\hspace{-5pt}#1}}
\newcommand{\KK}[1]{\overrightarrow{K\hspace{4pt}}_{\hspace{-5pt}#1}}
\newcommand{\BB}[1]{\overrightarrow{K\hspace{4pt}}_{\hspace{-5pt}#1}}
\newcommand{\PPP}[2]{\overrightarrow{P\hspace{4pt}}_{\hspace{-5pt}#1}({#2})}
\newcommand{\CCC}[2]{\overrightarrow{C\hspace{4pt}}_{\hspace{-5pt}#1}({#2})}

\newcommand{\bipart}[2]{#1 \Rightarrow #2}
\newcommand{\Dout}[1]{\mathcal{D}_{#1}}
\newcommand{\Lout}[1]{\mathcal{L}_{#1}}

\DeclareMathOperator{\ex}{ex}
\DeclareMathOperator{\exd}{ex_d}
\DeclareMathOperator{\exo}{ex_o}
\DeclareMathOperator{\dom}{dom}
\let\epsilon\varepsilon

\title{Tur\'an problems for oriented graphs}
\author{
Andrzej Grzesik\thanks{Faculty of Mathematics and Computer Science, Jagiellonian University, {\L}ojasiewicza 6, 30-348 Krak\'{o}w, Poland. E-mail: {\tt Andrzej.Grzesik@uj.edu.pl}. Supported by the National Science Centre grant 2021/42/E/ST1/00193.}\and Justyna Jaworska\thanks{Faculty of Mathematics and Computer Science, Jagiellonian University, {\L}ojasiewicza 6, 30-348 Krak\'{o}w, Poland. E-mail: {\tt justynajoanna.jaworska@student.uj.edu.pl}.}\and
Bart\l{}omiej Kielak\thanks{Faculty of Mathematics and Computer Science, Jagiellonian University, {\L}ojasiewicza 6, 30-348 Krak\'{o}w, Poland. E-mail: {\tt bartlomiej.kielak@doctoral.uj.edu.pl}.}\and Aliaksandra Novik\thanks{Section of Mathematics, École Polytechnique Fédérale de Lausanne, CH-1015 Lausanne, Switzerland. E-mail: {\tt aliaksandra.novik@epfl.ch}.}\and 
Tomasz \'Slusarczyk\thanks{Department of Mathematics, Massachusetts Institute of Technology, Cambridge, MA 02139, USA. E-mail: {\tt tomaszsl@mit.edu}.}}
\date{}

\theoremstyle{plain}
\newtheorem{theorem}{Theorem}

\newtheorem{conjecture}[theorem]{Conjecture}
\newtheorem{proposition}[theorem]{Proposition}
\newtheorem{claim}[theorem]{Claim}

\newtheorem{observation}[theorem]{Observation}
\newtheorem{problem}[theorem]{Problem}

\theoremstyle{definition}
\newtheorem{example}[theorem]{Example}
\newtheorem{definition}[theorem]{Definition}

\begin{document}

\maketitle

\begin{abstract}
A classical Tur\'an problem asks for the maximum possible number of edges in a graph of a~given order that does not contain a particular graph $H$ as a subgraph. It is well-known that the chromatic number of $H$ is the graph parameter which describes the asymptotic behavior of this maximum. 
Here, we consider an analogous problem for oriented graphs, where compressibility plays the role of the chromatic number. 
Since any oriented graph having a directed cycle is not contained in any transitive tournament, it makes sense to consider only acyclic oriented graphs as forbidden subgraphs. 
We provide basic properties of the compressibility, show that the compressibility of acyclic oriented graphs with out-degree at most 2 is polynomial with respect to the maximum length of a directed path, and that the same holds for a larger out-degree bound if the Erd\H os-Hajnal conjecture is true.
Additionally, generalizing previous results on powers of paths and arbitrary orientations of cycles, we determine the compressibility of acyclic oriented graphs with a restricted structure.
\end{abstract}

\section{Introduction}

For a~graph $H$, we denote by $\ex(n, H)$ the maximum possible number of edges in a~graph on $n$~vertices which does not contain $H$ as a~subgraph. The problem of determining the value of $\ex(n, H)$ for different graphs $H$ is one of the most fundamental questions in Extremal Graph Theory. 
Erd\H os and Stone \cite{ESS} found a tight asymptotic bound for $\ex(n, H)$ in terms of the chromatic number $\chi(H)$ of $H$ whenever $\chi(H) > 2$. Whereas for $\chi(H) = 2$, i.e., when $H$ is bipartite, no general bound is known and partial results include bounds for complete bipartite graphs \cite{KST54}, even cycles~\cite{BS74} or bipartite graphs with bounded degeneracy \cite{AKS03, GJN22}, for more details see the survey \cite{FS}.
The notion of $\ex(n, H)$ naturally generalizes to the setting when a family of graphs $\mathcal{H}$ is forbidden. 
It occurs that if we define $\chi(\mathcal{H})$ as the minimum of $\chi(H)$ for $H \in \mathcal{H}$, then Erd\H os-Stone Theorem still holds and gives tight asymptotic bounds for $\chi(\mathcal{H}) > 2$.

The notion of $\ex(n, H)$ can be defined similarly for directed graphs and oriented graphs. Recall that by a~\emph{directed graph} we mean a~pair $H = (V(H), E(H))$, where $V(H)$ is a~set of vertices and $E(H)$ is a~set of ordered pairs of different vertices called \emph{arcs}. Whereas by an \emph{oriented graph} we understand a~directed graph in which any two vertices are joined by at most one arc, in other words, an orientation of a graph.
To avoid ambiguity, we use notation $\exd$ in the setting of directed graphs and $\exo$ in the setting of oriented graphs.

Research on this problem in the directed setting can be traced back to the works of Brown, Erd\H os, Harary, H\"aggkvist, Simonovits, and Thomassen \cite{BES85, BES73, BH69, HT76}.  In particular, Brown, Erd\H os, and Simonovits \cite{BES73} proved that for every family of directed graphs $\mathcal H$ there exists a sequence $(G_n)_{n\geq1}$ of $n$-vertex graphs not containing any $H \in \mathcal H$ as a subgraph, such that each $G_n$ is a~blow-up of some fixed directed graph~$D$ and $\exd(n, \mathcal{H}) = \abs{E(G_n)} + o(n^2)$. Even though the theorem does not give much information about the graph $D$ itself, Valadkhan \cite{Valadkhan} observed that in the case of oriented graphs one may assume that $D$ is a tournament. It is clear that $D$ is the largest tournament whose blow-ups do not contain any $H \in \mathcal H$, which leads to the following crucial definition and theorem. 

\begin{definition}
The \emph{compressibility} of a family of oriented graphs $\mathcal H$, denoted by $\tau(\mathcal H)$, is the smallest $k \in \mathbb N$ such that for every tournament $T$ on $k$ vertices there exists $H \in \mathcal H$ which is homomorphic to $T$. If no such $k$ exists, we put $\tau(\mathcal H) = \infty$. For brevity, we define compressibility of an oriented graph $H$ as $\tau(H) := \tau(\{ H\})$. 
\end{definition}

\begin{theorem}[Valadkhan \cite{Valadkhan}]\label{thm:turan_for_oriented} 
For any family $\mathcal H$ of oriented graphs,
$$ \exo (n, \mathcal H) = \left(1 - \frac{1}{\tau(\mathcal{H})-1}\right)\binom{n}{2} + o(n^2).$$
\end{theorem}

Therefore, the compressibility plays the same role in the context of oriented graphs as the chromatic number in the context of graphs and the Erd\H{o}s-Stone Theorem. In particular, determining the compressibility of a graph or a family of graphs is asymptotically solving the respective problem on the maximum number of arcs in oriented graphs of a given order. 

Here, we focus on properties of $\tau(\mathcal{H})$ when $\mathcal{H}$ has a~single member. (In contrast to the chromatic number, in general $\tau(\mathcal{H})$ may differ from $\min \{\tau(H) : H \in \mathcal H\}$, see Example \ref{ex:compressibility_of_a_family}.) 
If an oriented graph contains a~directed cycle, and therefore it cannot be mapped homomorphically to any transitive tournament, then its compressibility is infinite. Therefore, we shall consider only acyclic oriented graphs. It is also easy to notice that a~transitive tournament on $k$~vertices does not contain a homomorphic image of any acyclic oriented graph with a directed path of order greater than $k$, hence $\tau(H)$ is always at least the maximum order $p(H)$ of a directed path in $H$. 
In fact, $\tau(H)$ can grow exponentially in terms of $p(H)$ as witnessed by transitive tournaments (Example \ref{ex:transitive_tournaments}) or particular orientations of complete bipartite graphs (Proposition \ref{prop:T_3-free_not_z_bounded}).
Therefore, it is natural to ask, as in~\cite{Valadkhan}, for which families of acyclic oriented graphs the growth is polynomial, or for which the trivial lower bound is optimal, i.e.,~$\tau(H)=p(H)$.

We show that the compressibility of acyclic oriented graphs with out-degree at most $2$ is polynomial with respect to the maximum order of a~directed path (Theorem \ref{thm:2-out-degree_bounded}), and that the same holds for a larger out-degree bound under the additional assumption that the Erd\H os-Hajnal conjecture holds (Theorem~\ref{thm:k-out-degree_bounded}).
Additionally, generalizing results for the square of a path, we determine the compressibility of acyclic oriented graphs with out-degree at most $2$ having restricted structure (Theorem~\ref{thm:path_with_steps}).
Finally, generalizing the result by Valadkhan \cite{Valadkhan} for acyclic orientations of cycles, we prove that the equality $\tau(H)=p(H)$ holds for oriented graphs $H$ with restricted distances of vertices to sinks and sources (Theorem~\ref{thm:l-layered}).

\section{Notation and basic properties of compressibility}

First, we shall introduce the notation used throughout the paper. Let $\TT{k}$ denote the transitive tournament on $k$ vertices. Let $\PP{k}$ be the \emph{directed path} on $k$ vertices, i.e.,~an orientation of a~path with all arcs directed towards the same end-point of the path. Similarly, let $\CC{k}$ be the \emph{directed cycle} on $k$ vertices, i.e.,~a~cyclic orientation of a~cycle of length $k$. Finally, let $\BB{s,t}$ denote the orientation of a~complete bipartite graph $K_{s,t}$ with all arcs directed towards the part of size $t$. If $G$ is an oriented graph and $v \in V(G)$, then we use the standard notation $d^+(v)$ and $d^-(v)$ for the out-degree and in-degree of a~vertex $v$ in $G$, respectively, and write $N^+(v)$ for the out-neighborhood of a~vertex $v$ in $G$. 

Let $G$ and $H$ be oriented graphs. By $G \odot H$ we mean the~\emph{composition} of oriented graphs, i.e., an~oriented graph created by replacing each vertex of $G$ by a~copy of $H$ and each arc of $G$ by $\BB{|H|,|H|}$ directed accordingly to the direction of the arc of $G$. Also, define $\bipart{G}{H}$ as the disjoint sum of $G$ and $H$ with all possible arcs from vertices of $G$ to vertices of $H$. In particular, if $G$ and $H$ are independent sets of size $s$ and $t$ respectively, then $\bipart{G}{H}$ is isomorphic to $\BB{s,t}$. 
By a~\emph{blow-up} of an oriented graph $G$ we mean a~graph created by replacing each vertex $v$ of $G$ by some independent set $I_v$ and each arc $uv$ of $G$ by $\BB{|I_u|,|I_v|}$. 
We say that $G$ is \emph{$H$-free} if $G$ does not contain a~subgraph isomorphic to $H$. If $\mathcal H$ is a~family of graphs, we say that $G$ is $\mathcal H$-free if $G$ is $H$-free for every $H \in \mathcal{H}$. If $H$ is a~subgraph of $G$ isomorphic to $H'$, we refer to $H$ as a~\emph{copy} of $H'$ in $G$. We write $H \to G$ if there exists a~homomorphism from $H$ to $G$, which is equivalent to saying that $H$ is a subgraph of some blow-up of $G$.

The compressibility of some particular graphs can be easily derived, for instance for directed paths.

\begin{example}\label{ex:paths}
For any $k \geq 1$, $\tau(\PP{k}) = k$, as every tournament on $k$ vertices contains a~copy of $\PP{k}$, i.e.,~a~Hamiltonian path, while there is no homomorphism $\PP{k} \to \TT{k-1}$.
\end{example}

If in the definition of compressibility we ask for the existence of an \emph{injective} homomorphism from $H$ to every tournament of a given order, then we obtain the definition of a \emph{1-color oriented Ramsey number}. See \cite{MT} for more information on this concept. As some graphs have no homomorphism into smaller oriented graphs, bounds on their compressibility follow from known bounds on their 1-color oriented Ramsey number.

\begin{example}\label{ex:transitive_tournaments}
Since the compressibility of a transitive tournament $\TT{k}$ is equal to the 1-color oriented Ramsey number of $\TT{k}$, standard probabilistic arguments \cite{EM, Stearns} imply that 
\[c_12^{k/2} \leq \tau(\TT{k}) \leq c_2 2^k\]
for some constants $c_1, c_2 > 0$ and any $k \geq 1$. These are essentially the best known general bounds.
\end{example}

In general, the compressibility of a family of graphs $\mathcal{H}$ can differ from $\min \{\tau(H) : H \in \mathcal H\}$ significantly. 

\begin{example}\label{ex:compressibility_of_a_family}
If $\mathcal H = \{\PP{2^k}, \TT{k}\}$ for any $k \geq 1$, then $\tau(\PP{2^k}) = 2^k$ and $\tau(\TT{k}) \geq c 2^{k/2}$ for some constant $c > 0$, but $\tau(\mathcal H) = k$, since each tournament $T$ on $k$ vertices either contains $\CC{3}$, and therefore there exists a~homomorphism $\PP{2^k} \to T$, or is transitive.
\end{example}

Let $p(H)$ be the order of a longest directed path in~$H$. By Example \ref{ex:paths}, $p(H)$ can be equivalently defined as the smallest $k$ for which there exists a~homomorphism $H \to \TT{k}$. In particular, Example~\ref{ex:transitive_tournaments} implies that the compressibility $\tau(H)$ is bounded exponentially in terms of $p(H)$. This motivates the following definition.

\begin{definition}\label{problem:delta-bounded}
Let $\mathcal{G}$ be a~family of acyclic oriented graphs. We say that $\mathcal{G}$ is \emph{polynomially $\tau$-bounded} if there exist constants $c, d$ such that for every $H \in \mathcal{G}$, we have
\[\tau(H) \leq c p(H)^d.\]
\end{definition}

Valadkhan \cite{Valadkhan} observed that containing a~large transitive tournament is not a necessary condition to have $\tau(H)$ exponentially large in terms of $p(H)$. Even forbidding $\TT{3}$ is not enough to guarantee polynomial $\tau$-boundedness.

\begin{proposition}[Valadkhan \cite{Valadkhan}]\label{prop:T_3-free_not_z_bounded}
For $n \geq 1$, let $H_n$ be the only acyclic orientation of $K_{n,n}$ such that $p(H_n) = 2n$. Then, $\tau(H_n) \geq 2^{n/2}$.
\end{proposition}

Note that if $\tau(H) = 2$, i.e.,~$H$ is a~subgraph of $\BB{s,t}$ for some $s,t \in \mathbb N$, Theorem \ref{thm:turan_for_oriented} implies only that $\exo(n, H) = o(n^2)$ and one may ask for the order of magnitude of $\exo(n, H)$. In some cases, $\exo(n, H)$ can be bounded by $c\cdot \ex(n, H')$ for some constant $c > 0$, where $H'$ is the graph obtained from $H$ by removing all orientations of arcs, hence the known bounds for $\ex(n, H')$ translate to the bounds for $\exo(n, H)$. In particular, K\H ovari-S\'os-Tur\'an Theorem \cite{KST54} gives the bound for $ \exo(\BB{s,t})$ for any $s,t \geq1$, while Bondy-Simonovits Theorem \cite{BS74} gives the bound for even cycles with edges oriented in alternating directions.

\section{Oriented graphs with bounded out-degree}\label{sec:bounded-degree}

For any integer $k \in \mathbb N$, let $\Dout{k}$ be the~family of all acyclic oriented graphs with out-degree bounded by $k$. In this section, we consider the question whether $\Dout{k}$ is polynomially $\tau$-bounded. 

Fox, He and Wigderson \cite[Theorem~1.4]{FHW21} showed (with a slight modification of their proof) that there exists a constant $c$ such that for every $H \in \Dout{k}$, it holds
\[\tau(H) \leq  (kp(H))^{ck\log p(H)}.\]
This means that for an acyclic oriented graph $H$ with bounded out-degree the compressibility $\tau(H)$ is quasi-polynomially bounded in terms of $p(H)$.
We prove that this can be improved to a polynomial bound if the following conjecture is true.

\begin{conjecture}\label{con:EH}
For every tournament $T$ there exists a constant $\epsilon>0$ such that every tournament on $n$ vertices contains either $T$ or a transitive tournament on $n^\epsilon$ vertices. 
\end{conjecture}

Alon, Pach, and Solymosi \cite{APS2001} proved that Conjecture \ref{con:EH} is equivalent to the well-known Erd\H os-Hajnal Conjecture \cite{EH89}.

\begin{theorem}\label{thm:k-out-degree_bounded}
Conjecture~\ref{con:EH} implies that $\Dout{k}$ is polynomially $\tau$-bounded for every $k \in \mathbb N$. 
\end{theorem}

Before we prove this theorem, let us introduce the following notion. For an oriented graph $H$, we say that a~subset $X \subseteq V(H)$ is \emph{dominated} in $H$ if $X \subseteq N^+(v)$ for some $v \in V(H)$. We have the following easy observation.

\begin{observation}\label{observation:k-dominated-subsets}
For any $k \geq 2$ and any tournament $T$, if all $k$-subsets of $V(T)$ are dominated in $T$, then for any $H \in \Dout{k}$ there exists a~homomorphism $H \to T$.
\end{observation}

\begin{proof}
Since $H$ is acyclic, there is an order of the vertices of $H$ in which all the arcs are directed backwards. We embed in $T$ the vertices of $H$ in this order using the fact that each vertex in $H$ has out-degree at most $k$ and each set of $k$ vertices in $T$ is dominated by some vertex of $T$.
\end{proof}

\begin{proof}[Proof of Theorem \ref{thm:k-out-degree_bounded}]
Our goal is to prove that for every $k \in \mathbb N$ there exists a tournament $T$ such that for each $H \in \Dout{k}$ there exists a homomorphism $H \to T$. If such $T$ exists, then Conjecture~\ref{con:EH} implies that for every $H \in \Dout{k}$, each tournament on $p(H)^{1/\epsilon}$ vertices contains a~copy of either $T$ or $\TT{p(H)}$. In both cases, it contains a homomorphic image of $H$. Thus, $\tau(H) \leq p(H)^{1/\epsilon}$.

Existence of such a tournament $T$ follows from a probabilistic argument. 
Let $n \in \mathbb N$ be large enough and $T$ be a~random tournament on $n$ vertices. For a~$k$-vertex subset $X \subseteq V(T)$, let $A_X$ be the event that $X$ is not dominated in $T$. Then, $A = \bigcup_{X \in \binom{V(T)}{k}} A_X$ is the event that some $k$-vertex subset of $V(T)$ is not dominated by any vertex. The probability of $A$ can be bounded as follows:
\begin{align*}
\mathbb P(A) & \leq  \sum_{X \in \binom{V(T)}{k}} \mathbb P(A_X) = \binom{n}{k} \left( 1 - \left(\frac 12\right)^k \right)^{n-k} \xrightarrow{n \to \infty} 0.
\end{align*}
Therefore, for large enough $n$, the probability of the complement of $A$ is positive, i.e.,~there exists a~tournament $T$ in which every set of $k$ vertices is dominated by some other vertex. By Observation~\ref{observation:k-dominated-subsets}, there exists a~homomorphism $H \to T$ for any $H \in \Dout{k}$.    
\end{proof}

In the case $k = 2$, one can notice that $\CC{3} \odot \CC{3}$ satisfies the assumption of Observation \ref{observation:k-dominated-subsets}. As \cite[Theorem~2.1]{APS2001} implies that the tournament $\CC{3} \odot \CC{3}$ satisfies Conjecture~\ref{con:EH} with the constant $\epsilon = 1/148$, we have
\[\tau(H) \leq cp(H)^{148}\] 
for any $H \in \Dout{2}$.
We prove a~much better bound.

\begin{theorem}\label{thm:2-out-degree_bounded}
There exists a~constant $c$ such that for every $H \in \Dout{2}$ we have
\[\tau(H) \leq c p(H)^4.\]
\end{theorem}

Before we prove this result, recall the notion of a domination graph.
The~\emph{domination graph} of a tournament $T$ is defined as the~spanning subgraph $\dom(T)$ of $T$ consisting of those arcs from $E(T)$ that are not dominated in $T$. One of the most basic properties of the domination graph is the following easy observation, the proof of which is included for completeness.

\begin{observation}\label{obs:arcs_in_dom(T)} 
If $T$ is a~tournament and $vw, v'w'$ are two vertex disjoint arcs from $E(\dom(T))$, then any arc between the sets $\{v,w\}$ and $\{v',w'\}$ completely determines the orientation of all the remaining arcs between those four vertices --- either $vv'$, $v'w$, $ww'$, $w'v \in E(T)$, or $vw'$, $w'w$, $wv'$, $v'v \in E(T)$.
\end{observation}

\begin{proof}
If the arcs of the tournament $T$ between the sets $\{v,w\}$ and $\{v',w'\}$ are not forming a directed cycle, then there exists a vertex such that either $v$ and $w$ or $v'$ and $w'$ are its out-neighbors, which contradicts the fact that arcs $vw$ and $v'w'$ are not dominated. 
Thus, the arcs between the sets $\{v,w\}$ and $\{v',w'\}$ are forming a directed cycle. 
Depending on its direction, we obtain one of the two possibilities listed in the statement of the observation. 
\end{proof}

We are ready now to prove Theorem \ref{thm:2-out-degree_bounded}.

\begin{proof}[Proof of Theorem \ref{thm:2-out-degree_bounded}]
We use induction on $p(H)$. If $p(H)=2$, then $\tau(H)=2$. If $p(H)=3$, then $H \to \TT{3}$, and since any tournament on $4$ vertices contains $\TT{3}$, we have $\tau(H) \leq 4$. Thus, for $p(H) \leq 3$ the inequality $\tau(H) \leq c p(H)^4$ holds for any $c \geq 2$. 

Let $T$ be any tournament on $cp(H)^{4}$ vertices for some constant $c > 0$ and $p(H) \geq 4$. As there is a homomorphism from $H$ to $\TT{p(H)}$, we may assume that $T$ does not contain a transitive tournament on $p(H)$ vertices. 

Assume first that $\dom(T)$ does not contain a~matching on $cp(H)^{3}$ vertices. 
By removing from~$T$ the vertices of any maximum matching in $\dom(T)$, we obtain a~tournament $T'$ on at least $cp(H)^4 - 2cp(H)^3$ vertices, which is greater than $c(p(H)-1)^{4}$ for $p(H) \geq 4$. If we let $H'$ be the subgraph of $H$ obtained by removing all sources in $H$, then $p(H') = p(H) - 1$ and we can apply the induction hypothesis to find a~homomorphism $H' \to T'$. Since every pair of vertices from $V(T')$ is dominated in $T$ and the maximum out-degree of $H$ is at most two, we can extend this homomorphism to $H \to T$. Therefore, we may assume that there exists a~subgraph $M$ of $\dom(T)$ which is a~matching on at least $cp(H)^{3}$ vertices.

From Observation \ref{obs:arcs_in_dom(T)}, it follows that for every arc $vw \in E(M)$ and every other vertex \hbox{$u \in V(M)$}, either $vu, uw \in E(T)$ or $wu, uv \in E(T)$. Therefore, if we pick one vertex from each arc in $E(M)$ and denote by $T_M$ the subtournament of $T$ induced by those vertices, then $T_M$ can be considered as equipped with a~special operation of \emph{flipping} a~vertex, i.e.,~reversing the orientations of all arcs incident to this vertex. Indeed, this operation corresponds to replacing this vertex by its neighbor in $M$.

We want to prove that there exists a~subgraph of $T_M$ isomorphic to $\CC{3} \odot \CC{3}$, because this implies that $H \to T$ by Observation \ref{observation:k-dominated-subsets}. Note that $\CC{3} \odot \CC{3}$ consists of three clusters, each being a~copy of~$\CC{3}$. If we flip all vertices from one cluster, then this cluster will remain a~copy of~$\CC{3}$, but arcs between this cluster and remaining ones will reverse, resulting in a~subgraph isomorphic to $\TT{3} \odot \CC{3}$. 
Therefore, it is enough to prove that every tournament on $cn^{3}$ vertices contains a copy of $\TT{3} \odot \CC{3}$ or $\TT{n}$. 
As $\TT{3} \odot \CC{3}$ is isomorphic to $\bipart{(\bipart{\CC{3}}{\CC{3}})}{\CC{3}}$, we force its appearance in two steps using the following claim.

\begin{claim}
For any oriented graph $D$ and constants $c_0, \delta > 0$, if every $\TT{n}$-free tournament on $c_{0}n^{\delta}$ vertices contains a~copy of $D$, then there exists $c > 0$ such that every $\TT{n}$-free tournament on $cn^{\delta+1}$ vertices contains a~copy of $\bipart{D}{\CC{3}}$. 
\end{claim}

\begin{proof}
Let $T'$ be any $\TT{n}$-free tournament on $an^{\delta+1}$ vertices for $a\geq \max(3\sqrt[4]{8c_0}, 6)$. Assume additionally that $T'$ contains at most $n^{3\delta+2}$ copies of~$\CC{3}$. 
Since we need to find a copy of $\bipart{D}{\CC{3}}$, we want to find a lower bound for the number $t'$ of copies of~$\bipart{\TT{1}}{\CC{3}}$ in $T'$. 
As every tournament on $an^{\delta+1}$ vertices contains at least $an^{\delta+1}/3$ vertices of out-degree at least $an^{\delta+1}/3$ (otherwise not all vertices of smaller degree could be connected), we may choose the source of $\bipart{\TT{1}}{\CC{3}}$ among those $an^{\delta+1}/3$ vertices. 
Now, since every  $\TT{n}$-free tournament on $2n$ vertices contains at least $n$ copies of $\CC{3}$, we can count the number of subsets of size $2n$ in the out-neighborhood restricted to $\lceil an^{\delta+1}/3 \rceil$ vertices and obtain
\[t' \geq \frac{an^{\delta+1}}{3} \cdot \frac{n\binom{\lceil an^{\delta+1}/3 \rceil}{2n}}{\binom{\lceil an^{\delta+1}/3 \rceil -3}{2n-3}} = \frac{an^{\delta+2}}{3} \cdot \frac{\binom{\lceil an^{\delta+1}/3 \rceil}{3}}{\binom{2n}{3}} \geq \frac{an^{\delta+2}}{3} \cdot \frac{(an^{\delta+1})^3}{(2n)^3\cdot 3^3} = \frac{a^4 n^{4\delta+2}}{2^3\cdot3^4},\]
as every copy of~$\bipart{\TT{1}}{\CC{3}}$ will be counted this way at most $\binom{\lceil an^{\delta+1}/3\rceil -3}{2n-3}$ times.
Since there are at most $n^{3\delta+2}$ copies of~$\CC{3}$ in $T'$, there exists a~copy of~$\CC{3}$ which is dominated by at least \[\frac{t'}{n^{3\delta+2}} \geq \frac{a^4}{2^3 \cdot 3^4} n^{\delta} \geq c_0n^{\delta}\] 
vertices of $T'$. Since any subtournament of $T'$ of order at least $c_0n^\delta$ contains a~copy of $D$, we conclude that the tournament $T'$ contains the desired copy of $\bipart{D}{\CC{3}}$.

In order to prove the claim, consider any~$\TT{n}$-free tournament $T$ on $cn^{\delta+1}$ vertices for some $c\geq \max(3^4 a^3c_0, 3a)$. From the previous paragraph, we may assume that every subtournament on $an^{\delta+1}$ vertices contains at least $n^{3\delta+2}$ copies of~$\CC{3}$. 
By the same counting argument, we get that the number $t$ of copies of~$\bipart{\TT{1}}{\CC{3}}$ in $T$ satisfies
\[t \geq \frac{cn^{\delta+1}}{3} \cdot \frac{n^{3\delta+2}\binom{\lceil cn^{\delta+1}/3\rceil}{an^{\delta+1}}}{\binom{\lceil cn^{\delta+1}/3\rceil -3}{an^{\delta+1}-3}} = \frac{cn^{4\delta + 3}}{3} \cdot \frac{\binom{\lceil cn^{\delta+1}/3\rceil}{3}}{\binom{an^{\delta+1}}{3}} \geq \frac{cn^{4\delta+3}}{3} \cdot \frac{(cn^{\delta+1})^3}{(an^{\delta+1})^3\cdot 3^3} = \frac{c^4 n^{4\delta+3}}{a^3\cdot3^4}.\]
Since there are at most $c^3n^{3\delta+3}$ copies of~$\CC{3}$ in $T$, there exists a~copy of~$\CC{3}$ that is dominated by at least \[\frac{t}{c^3n^{3\delta+3}} \geq \frac{c}{a^3\cdot 3^4} \cdot n^{\delta} \geq c_0n^{\delta}\]
vertices of $T$. 
Thus, $T$ contains the desired copy of $\bipart{D}{\CC{3}}$.
\end{proof}

Applying the above claim for $n=p(H)$, $D = \CC{3}$, $\delta = 1$, and $c_0 > 1$, and afterwards for $D = (\bipart{\CC{3}}{\CC{3}})$ and $\delta = 2$ we conclude that the tournament $T_M$ on $cp(H)^3$ vertices contains a~copy of $\TT{3} \odot \CC{3}$ or $\TT{p(H)}$, which ends the proof of Theorem~\ref{thm:2-out-degree_bounded}.
\end{proof}

For certain subclasses of $\Dout{k}$, it is possible to find homomorphisms into tournaments of even linear order. For instance, Dragani\'c et al. proved the following result for powers of paths. 

\begin{theorem}[Dragani\'c et al.~\cite{10x2021}]\label{thm:kth-powers-of-paths}
For every $n, k \geq 2$, every tournament on $n$ vertices contains a~$k$-th power of a~directed path of order $n / 2^{4k+6}k + 1$. Moreover, for $k=2$, every tournament on $n$ vertices contains a~square of a~directed path of order $\lceil 2n/3 \rceil$ and this value is optimal.
\end{theorem}

A square of a directed path, considered in Theorem~\ref{thm:kth-powers-of-paths}, is an oriented graph obtained from a directed path by adding arcs between vertices at distance $2$. A~generalization of this structure is an oriented graph obtained from a directed path by adding arcs between vertices at some different distance. 

\begin{definition}\label{defi:C_k_l}
For any $2 \leq \ell < k$, let $\PPP{k}{\ell}$ be the oriented graph on $k$ vertices $v_1, \ldots, v_k$ with arcs $v_iv_{i+1}$ for $1 \leq i \leq k-1$ and $v_iv_{i+\ell}$ for $1 \leq i \leq k-\ell$. In other words, $\PPP{k}{\ell}$ is a~directed path on $k$ vertices with additional arcs between vertices at distance $\ell$. Let also $\CCC{k}{\ell}$ be the oriented graph on $k$ vertices $w_0, w_1, \ldots, w_{k-1}$ with arcs $w_iw_{i+1 \pmod k}$ and $w_iw_{i+\ell \pmod k}$ for $0 \leq i < k$.
\end{definition}

As $\PPP{k}{\ell}$ is a subgraph of the $\ell$-th power of $\PP{k}$, Theorem~\ref{thm:kth-powers-of-paths} implies that $\tau(\PPP{k}{\ell})$ is linear in terms of $p(\PPP{k}{\ell}) = k$. But the constant provided in Theorem~\ref{thm:kth-powers-of-paths} for large $\ell$ is very far from being optimal. 
The following theorem closes this gap and shows that for $\ell=2$ and $3$, the compressibility of $\PPP{k}{\ell}$ differs from the compressibility of $\PP{k}$.

\begin{theorem}\label{thm:path_with_steps}
For every $2 \leq \ell < k$, the following holds
\begin{itemize}[itemsep=0pt, topsep=2pt]
\item $\tau(\PPP{k}{2}) = \left\lfloor\frac{3k-1}{2}\right\rfloor$,
\item $\left\lfloor\frac{7k-1}{6}\right\rfloor \leq \tau(\PPP{k}{3}) \leq 3k$,
\item $\tau(\PPP{k}{\ell}) = k$ if $\ell \geq 4$.
\end{itemize}   
\end{theorem}

\begin{proof}
For $\ell=2$, the graph $\PPP{k}{2}$ is just a~square of a~path, and Theorem~\ref{thm:kth-powers-of-paths} implies that every tournament on $\left\lfloor\frac{3k-1}{2}\right\rfloor$ vertices contains a~copy of $\PPP{k}{2}$. On the other hand, there are tournaments on $\left\lfloor\frac{3k-1}{2}\right\rfloor-1$ vertices that do not have a homomorphism from $\PPP{k}{2}$. For odd~$k$, we consider the tournament $\PP{(k-1)/2} \odot \CC{3}$, while for even $k$ consider the tournament $\bipart{\TT{1}}{(\PP{k/2-1} \odot \CC{3})}$. The considered tournaments have exactly $\left\lfloor\frac{3k-1}{2}\right\rfloor-1$ vertices and any homomorphism of $\PP{k}$ into them maps some three consecutive vertices into a~copy of $\CC{3}$, which cannot happen for the homomorphism of $\PPP{k}{2}$. 

If $\ell \geq 4$, then $\tau(\PPP{k}{\ell}) \geq k$ as there exists no homomorphism $\PPP{k}{\ell} \to \TT{k-1}$. To prove the upper bound, consider any tournament $T$ on $k$ vertices. Then, $T$ admits a~decomposition $T_1 \Rightarrow \ldots \Rightarrow T_m$ into strongly connected components. If any of those components is of size at least $\ell-1$, then it contains a~copy of $\CC{\ell}$, and since there is a~homomorphism $\PPP{k}{\ell} \to \CC{\ell-1}$, we have $\PPP{k}{\ell} \to T$. Otherwise, all strongly connected components are of size strictly smaller than $\ell-1$. This means that any function that maps the Hamiltonian path of $\PPP{k}{\ell}$ into any Hamiltonian path of $T$ induces a~homomorphism $\PPP{k}{\ell} \to T$. 

We are left with the hardest case $\ell = 3$.
To prove the lower bound, consider a~tournament $\widetilde T$ on 7 vertices $v_1, \ldots, v_7$, with arcs $v_iv_j$ for $1 \leq i < j \leq 6$ and $N^+(v_7) = \{v_1, v_2, v_4\}$, see Figure~\ref{fig:T7}. 
We want to prove that there exists no homomorphism $\PPP{7}{3} \to \widetilde T$. This implies that there exists no homomorphism of $\PPP{6a+1}{3} \to \TT{a} \odot \widetilde T$ for any integer $a \geq 1$ and the claimed lower bound follows.

\begin{figure}[ht]
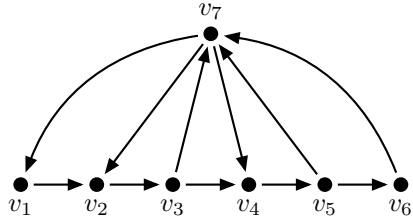

\centering
\figureTg
\caption{Tournament $\widetilde T$ from the proof of Theorem \ref{thm:path_with_steps} for $\ell=3$. The bottom vertices induce a~transitive tournament.}
\label{fig:T7}
\end{figure}

Assume that $x_1, x_2, \ldots, x_7$ are the images of consecutive vertices of $\PPP{7}{3}$ under some homomorphism $\PPP{7}{3} \to \widetilde T$. As the vertices $v_1, \ldots, v_6$ induce a transitive tournament, there must exist the smallest~$i$ such that $x_i = v_7$. 
If $i=1$, then since $x_1x_4$ is an arc and $x_1x_2x_3x_4$ is a path, we must have $x_4 = v_4$. But then it is not possible to find a path  $x_4x_5x_6x_7$ with an arc $x_4x_7$.
If $2 \leq i \leq 4$, then similarly $x_{i+3} = v_4$, hence $x_{i+2} \in \{v_1, v_2, v_3\}$. But since $x_{i-1}x_i$ is an arc, we have $x_{i-1} \in \{v_3, v_5, v_6\}$ and it is not possible for $x_{i-1}x_{i+2}$ to be an arc.
If $5 \leq i \leq 6$, then by a~symmetric argument we conclude that $x_{i-3} = v_3$, $x_{i-2} \in \{v_4, v_5, v_6\}$ and $x_{i+1} \in \{v_1, v_2, v_4\}$, hence $x_{i-2}x_{i+1}$ cannot be an arc.
Finally, if $i = 7$, then we must have $x_j = v_j$ for every $1 \leq j \leq 7$, but in this case $x_4x_7$ is not an arc. This finishes the proof of the lower bound. 

In order to prove the upper bound, we apply the following theorem that characterizes the general structure of the domination graphs of tournaments.
Here, by a~\emph{directed caterpillar} we mean a~directed path with possible additional outgoing pendant arcs. 

\begin{theorem}[Fisher et al.~\cite{FLMR95}]\label{thm:domination_graphs}
The domination graph of a tournament is either an odd directed cycle with possible outgoing pendant arcs and isolated vertices, or a~forest of directed caterpillars.
\end{theorem}

We prove by induction on $k$ that for every tournament on $3k$ vertices there exists a~homomorphism from $\PPP{k}{3}$. For $k \leq 3$, an oriented graph $\PPP{k}{3}$ is just a directed path $\PP{k}$, which can be mapped homomorphically into any tournament on $k$ vertices (Example~\ref{ex:paths}).
 
For $k > 3$, let $T$ be any tournament on $3k$ vertices. Note that $\PPP{k}{3} \to \CCC{5}{3}$, so we may assume that $T$ does not contain $\CCC{5}{3}$. Denote vertices of $\PPP{k}{3}$ by $w_1, \ldots, w_k$ with arcs of the form $w_iw_{i+1}$ and $w_iw_{i+3}$. 
Whenever we use the induction hypothesis to obtain a homomorphism $\PPP{k-1}{3} \to T$, we think of this $\PPP{k-1}{3}$ as of a~subgraph of $\PPP{k}{3}$ induced by vertices $w_2, w_3, \ldots, w_k$. In particular, in order to find a homomorphism $\PPP{k}{3} \to T$, we only need to map $w_1$ to a vertex  dominating the images of $w_2$ and $w_4$. This is possible exactly when the images of $w_2$ and $w_4$ induce an arc which does not belong to $E(\dom(T))$.

It turns out that if $\dom(T)$ contains a~cycle of length at least five, two caterpillars, or a caterpillar with a directed path of length at least three, then $T$ must contain $\CCC{5}{3}$. It follows from the following observation.

\begin{observation}\label{obs:existence_of_C_2_5}
If $\dom(T)$ contains two vertex disjoint arcs, whose sources are not connected by an arc in $\dom(T)$, then $T$~contains a~copy of $\CCC{5}{3}$. 
\end{observation}

\begin{proof}
Let $vw$ and $v'w'$ be the two arcs in $\dom(T)$, and without loss of generality let $v'v \in E(T) \setminus E(\dom(T))$. 
By Observation \ref{obs:arcs_in_dom(T)}, all arcs between $vw$ and $v'w'$ are then completely determined. Moreover, since $v'v \not \in E(\dom(T))$, there exists a~vertex $u$ which dominates $v'v$, in particular it is neither $w$ nor $w'$. Since $vw$ and $v'w'$ are not dominated, we have that $wu$, $w'u \in E(T)$. Now, it is straightforward to check that vertices $v$, $w$, $v'$, $w'$ and $u$, in this order, induce a~copy of $\CCC{5}{3}$, as depicted in Figure \ref{fig:C5}.
\end{proof}

\begin{figure}[ht]
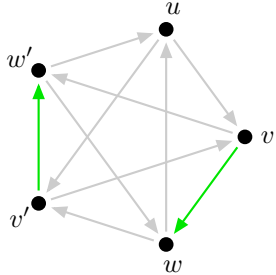

\centering
\figurezero
\caption{$\CCC{5}{3}$ created in $T$ using Observation~\ref{obs:existence_of_C_2_5}. Green arcs belong to $E(\dom(T))$.}
\label{fig:C5}
\end{figure}

By Theorem~\ref{thm:domination_graphs} and Observation \ref{obs:existence_of_C_2_5}, $\dom(T)$ must be either a~directed triangle with some outgoing arcs or a~directed caterpillar with a longest directed path of length at most $2$. In particular, there exist at most three vertices with a~positive out-degree in $\dom(T)$, hence it is possible to find a~subset $D \subseteq V(T)$ of size at most $3$ such that each arc from $E(\dom(T))$ is incident to at least one vertex from $D$. Let $T'$ be the subtournament of $T$ induced by $V(T) \setminus D$. Since $\abs{V(T')} \geq 3(k-1)$, by the induction hypothesis there exists a homomorphism $\PPP{k-1}{3} \to T'$. Moreover, the arc induced by the images of $w_2$ and $w_4$ cannot belong to $E(\dom(T))$, hence we can extend this homomorphism to $\PPP{k}{3} \to T$.
\end{proof}

\section{Compressibility of $\ell$-layered graphs}

In this section, we study a class of acyclic oriented graphs $H$ for which $\tau(H)$ = $p(H)$. The considered class contains in particular graphs $\PPP{k}{\ell}$ for $\ell\geq4$, for which the equality holds by Theorem~\ref{thm:path_with_steps}, as well as some graphs with out-degree not bounded by $p(H)$. It also generalizes the results of Valadkhan \cite{Valadkhan} for orientations of trees and cycles. 

\begin{definition}
We say that an acyclic oriented graph $H$ is \emph{$\ell$-layered} if for every vertex $v \in V(H)$ which is not a~sink nor a~source there exists a~pair $(i,j) \in \mathbb Z_\ell^2$ such that the length of every directed path from any source of $H$ to $v$ is congruent to $i$ modulo $\ell$ and the length of every directed path from $v$ to any sink of $H$ is congruent to $j$ modulo $\ell$. If a~vertex $v$ was assigned a~pair $(i,j)$, we will say that it is of type $(i,j)$. 
\end{definition}

For $\ell \geq 2$, let $\Lout{\ell}$ denote the family of all $\ell$-layered acyclic oriented graphs. 

\begin{example} For any $3 \leq \ell < k$, the graph $\PPP{k}{\ell}$ is $(\ell-1)$-layered.
\end{example}

\begin{example}
Consider an acyclic oriented graph $H$ and some integer $\ell \geq 2$, and replace each arc $uv$ of $H$ by a~directed path of length $\ell$ from $u$ to $v$. Then, the resulting graph, also called an $(\ell-1)$-subdivision of $H$, is $\ell$-layered.
\end{example}

\begin{example}
For any integers $k \geq 3$ and $\ell \geq 2$, each acyclic orientation of a~cycle on $k$ vertices is $\ell$-layered. 
\end{example}

\begin{example}
An acyclic oriented graph obtained from a directed path $v_1v_2\ldots v_k$ by adding a~new vertex $v$ and an arc $v_{k-2}v$ is not $\ell$-layered for any $\ell \geq 2$. It follows from the fact that the distance from $v_1$ to $v$ is $k-2$, while from $v_1$ to $v_k$ it is $k-1$.
On the other hand, it is easy to observe that any acyclic orientation of a~tree can be mapped homomorphically to some directed path, which is $\ell$-layered for every $\ell \geq 2$.
\end{example}

Since the oriented graph in Proposition \ref{prop:T_3-free_not_z_bounded} is $2$-layered, the class $\Lout{2}$ is not polynomially $\tau$-bounded. However, for $\ell \geq 3$ the situation is completely different.

\begin{theorem}\label{thm:l-layered}
Let $\ell \geq 3$ and $H \in \Lout{\ell}$ with $p(H) \geq 6$. Then, $\tau(H) = p(H)$. 
\end{theorem}

\begin{proof}
Firstly, observe that $H$ can be mapped homomorphically into $\bipart{\CC{\ell}}{\TT{1}}$. Indeed, if we denote the consecutive vertices of $\CC{\ell}$ by $w_0, w_1, \ldots, w_{\ell-1}$ and the only vertex of $\TT{1}$ by $w$, then we can define a~map $H \to \bipart{\CC{\ell}}{\TT{1}}$ in the following way: assign every source of $H$ to $w_0$, every sink of $H$ to $w$, and every vertex of type $(i,j)$ to $w_i$. It is straightforward to check that this is indeed a~homomorphism.

If $T'$ is any tournament on $5$ vertices containing a~copy of~$\CC{5}$, then some vertex of $T'$ is contained in a~copy of $\CC{3}$ and a~copy of $\CC{4}$. 
Thus, there is a homomorphism $\CC{\ell} \to T'$ for any $\ell \geq 3$.
In particular, there always exists a~homomorphism $H \to \bipart{T'}{\TT{1}}$. An analogous argument shows that there always also exists a~homomorphism $H \to \bipart{\TT{1}}{T'}$.

Fix now a~tournament $T$ on $p(H)$ vertices. Assume that $T$ is not strongly connected. If at least one strongly connected component is of size at least $\min(5, \ell)$, then there exists a~homomorphism $H \to T$ by the observation above. Therefore, we may assume that all strongly connected components of $T$ are of size smaller than $\min(5, \ell)$. 
For each $v \in V(H)$, let $\ell(v)$ denote the length of any longest directed path in $H$ starting at $v$. Choose any Hamiltonian path $P$ in $T$ with vertices in order $v_{p(H)-1}, \ldots, v_0$. Since every strongly connected component of $T$ is of size smaller than $\ell$, we have $v_iv_j \in E(T)$ for any $i - j > \ell$. Define a~map $H \to T$ by assigning each $v \in V(H)$ to $v_{\ell(v)}$. Since for each arc $vw \in E(H)$ we have either $\ell(v) - \ell(w) = 1$ or $\ell(v) - \ell(w) > \ell$, it follows that this map is indeed a~homomorphism.

Since any strongly connected tournament on $p(H)$ vertices contains a~strongly connected subtournament on $6$ vertices, it is enough to show that there exists a~homomorphism from $H$ to any strongly connected tournament on $6$ vertices.

Let us introduce the following tournaments on $5$ vertices:
\begin{itemize}
\item $T_a$, obtained from $\CCC{5}{3}$ by reversing the arc $w_1w_4$;
\item $T_b$, obtained from $\CCC{5}{2}$ by reversing the arc $w_1w_4$;
\item $T_c$, obtained from $\TT{5}$ by reversing the arc between the sink and the source;
\item $T_d$, obtained from $T_c$ by reversing the arc $w_3w_5$;
\item $T_e$, obtained from $T_c$ by reversing the arc $w_2w_4$. 
\end{itemize}
All of them are depicted in Figure \ref{fig:tournaments-3-labeled}. Let $\mathcal T = \{T_a, T_b, T_c, T_d, T_e\}$.
By showing a series of claims we will prove that every strongly connected
tournament on 6 vertices contains some tournament from $\mathcal T$, and that there exists a homomorphism from $H$ to any tournament in $\mathcal T$.

\begin{figure}[ht]
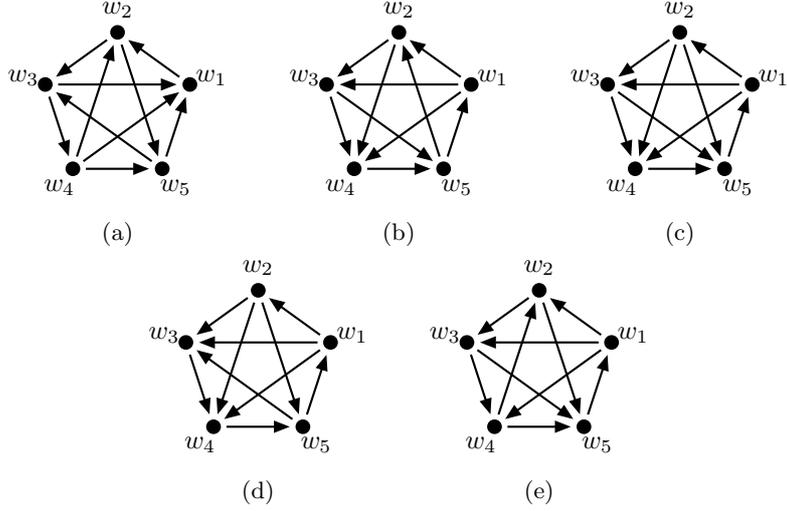

\centering
\begin{subfigure}[b]{0.24\linewidth}
\centering
\figureTa 
\caption{}\label{fig:Ta}
\end{subfigure}
\begin{subfigure}[b]{0.24\linewidth}
\centering
\figureTb 
\caption{}\label{fig:Tb}
\end{subfigure}
\begin{subfigure}[b]{0.24\linewidth}
\centering
\figureTc 
\caption{}\label{fig:Tc}
\end{subfigure}

\begin{subfigure}[b]{0.24\linewidth}
\centering
\figureTd 
\caption{}\label{fig:Td}
\end{subfigure}
\begin{subfigure}[b]{0.24\linewidth}
\centering
\figureTe 
\caption{}\label{fig:Te}
\end{subfigure}
\caption{Tournaments used in the proof of Theorem \ref{thm:l-layered}.}
\label{fig:tournaments-3-labeled}
\end{figure}

\begin{claim}\label{clm:5-vertex-tournaments}
Every strongly connected tournament on $5$ vertices is isomorphic to $\CCC{5}{2}$ or some $T \in \mathcal T$.
\end{claim}
\begin{proof}
Let $T$ be a~strongly connected tournament on $5$ vertices $w_1, \ldots, w_5$ with arcs $w_5w_1$ and $w_iw_{i+1}$ for $1 \leq i \leq 4$. If there are no vertices in $T$ with out-degree equal to $3$, then $d^+(w_i) = 2$ for every $1 \leq i \leq 5$ and $T$ is isomorphic to $\CCC{5}{2}$ (since $\CCC{5}{2}$ and $\CCC{5}{3}$ are isomorphic).

Assume now that there is exactly one vertex $v$ in $T$ with out-degree $3$. Then, there is also exactly one vertex $w$ with out-degree $1$. If $vw \in E(T)$, then by reversing an arc $vw$ we obtain a~tournament $T'$ with all vertices having out-degree $2$, hence $T'$ is isomorphic to $\CCC{5}{3}$ and $T$ is isomorphic to either $T_a$ or $T_b$. If $wv \in E(T)$, then the three remaining vertices of $T$ are in out-neighborhood of $v$ and in-neighborhood of $w$. They must induce a~copy of $\CC{3}$, since $v$ is the only vertex with out-degree $1$. But then, $T$ is isomorphic to $T_e$.

We are left with the case when there are two vertices with out-degree $3$. It is easy to see that they must be neighbors in a~copy of $\CC{5}$ contained in $T$, which determines all but one arc in $T$. Depending on the orientation of this remaining arc, we conclude that $T$ is isomorphic either to $T_c$ or to $T_d$. 
\end{proof}

\begin{claim}\label{clm:T-in-6-vertex}
Every strongly connected tournament on $6$ vertices contains a~copy of some $T \in \mathcal T$.
\end{claim} 
\begin{proof}
By Claim \ref{clm:5-vertex-tournaments}, it is enough to find a~strongly connected subtournament with a~vertex of in-degree or out-degree equal to $3$.
Let $T$ be any strongly connected tournament on $6$ vertices. It must contain a~copy of $\CC{5}$ and vertices of this copy induce a~strongly connected subtournament $T'$. If $T'$ is isomorphic to some element of $\mathcal T$, then we are done. Otherwise, by Claim \ref{clm:5-vertex-tournaments}, it must be isomorphic to $\CCC{5}{2}$; let $w_1, \ldots, w_5$ be consecutive vertices of the outer directed cycle of $T'$, and let $w$ denote the remaining vertex of $T$. Since $T$ is strongly connected, $w$ has in-neighbors and out-neighbors in $T'$; without loss of generality, we may assume that $w_1w, ww_2 \in E(T)$. If $ww_4 \in E(T)$, then the subtournament $T_1$ induced by vertices $w, w_2, w_3, w_4, w_1$ is strongly connected and in-degree of $w_4$ in $T_1$ is equal to $3$. If $w_4w \in E(T)$, then the~subtournament $T_2$ induced by vertices $w, w_2, w_4, w_5, w_1$ is strongly connected and out-degree of $w_4$ is equal to $3$. In both cases, $T_1$ or $T_2$ is isomorphic to some element of $\mathcal T$, which finishes the proof. 
\end{proof}

To simplify the proof that $H$ has a homomorphism to each $T\in\mathcal T$, we want to construct an oriented graph $Q_\ell$ such that $H$ can be mapped homomorphically into $Q_\ell$ and then for each $T$ provide a homomorphism from $Q_\ell$. For every $0 \leq i < \ell$, let $D_i$ be a~directed cycle on a~vertex set $\{(j, i-j)  \in \mathbb Z_\ell^2: 0 \leq j < \ell\}$ with arcs from $(j, i-j)$ to $(j+1, i-j-1)$ for every $0 \leq j < \ell$. Define $Q_\ell$ as a~disjoint union of $D_i$, over all $0 \leq i < \ell$, and two additional vertices $v_s$, $v_t$, with arcs joining $v_s$ to $v_t$, $v_s$ to $(1,i)$, and $v_t$ from $(i,1)$ for all $0 \leq i < \ell$. Since the graph $H$ is $\ell$-layered, we have a~natural homomorphism $H \to Q_\ell$ which maps all sources of $H$ to $v_s$, all sinks of $H$ to $v_t$, and all vertices of type $(i,j)$ to the vertex $(i,j)$ for every pair $(i,j) \in \mathbb Z_\ell^2$.  

\begin{claim}\label{clm:3-4-layered-gadget}
Let $T$ be a~tournament on at least $5$ vertices. Assume there exist vertices $u, v \in V(T)$ such that $uv \in E(T)$ and:
\begin{itemize}
\item $vw, wu, uz, zv \in E(T)$ for some $w, z \in V(T)$ and $w$ is contained in a~copy of $\CC{3}$,
\item $ux, xy, yv \in E(T)$ for some $x, y \in V(T)$ and an arc $xy$ is contained in a~copy of $\CC{3}$ and $\CC{4}$.
\end{itemize}
If $\ell=3$ or $\ell=4$, then there exists a~homomorphism $Q_\ell \to T$.
\end{claim}

\begin{proof}
Start defining the homomorphism $Q_\ell \to T$ by assigning $v_s$ to $u$ and $v_t$ to $v$. It remains to define homomorphism $D_i \to T$ for every $0 \leq i < \ell$ such that the image of $(1,i)$ is in out-neighborhood of $u$ and the image of $(i,1)$ is in the in-neighborhood of $v$. Assign $(1,1)$ to $z$, $(1,2)$ to $x$, and $(2,1)$ to $y$. If $\ell = 3$, then assign $(0,1)$ to $u$ and $(1,0)$ to $v$. If $\ell=4$, then assign $(0,1)$ and $(1,1)$ to $z$, $(3,1)$ to $u$, and $(1,3)$ to $v$. All of these assignments are depicted in Figure \ref{fig:layered-3-4-gadget}. Using the assumptions in the claim, it is straightforward to check that this can be extended to a~homomorphism $Q_\ell \to T$.
\end{proof}

\begin{figure}[ht]
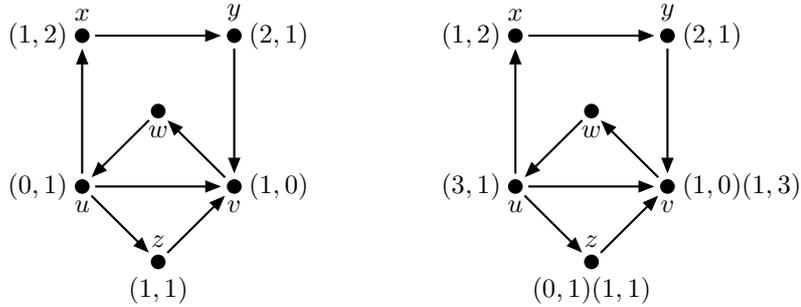

\centering
\begin{subfigure}[b]{0.4\linewidth}
\centering
\figureTlaygadget
\end{subfigure}
\begin{subfigure}[b]{0.4\linewidth}
\centering
\figureTlaygadgettwo 
\end{subfigure}
\caption{Partial homomorphisms $Q_3 \to T$ and $Q_4 \to T$ from the proof of Claim \ref{clm:3-4-layered-gadget}. The arc $xy$ is assumed to be contained in a~copy of $\CC{4}$ and $\CC{3}$, while the~vertex $z$ is assumed to be contained in a~copy of $\CC{3}$.}
\label{fig:layered-3-4-gadget}
\end{figure}

\begin{claim}\label{clm:Q_l-T-3-5}
For every $3 \leq \ell \leq 5$ and every $T \in \mathcal T$, there exists a~homomorphism $Q_\ell \to T$.
\end{claim}

\begin{proof}
If $\ell=3$ or $\ell = 4$, it is enough for every $T \in \mathcal T$ to find vertices $u, v \in V(T)$ satisfying the assumptions of Claim \ref{clm:3-4-layered-gadget}. It is easy to verify that one can choose:
\begin{itemize}[itemsep=0pt]
\item $w_4$ as $u$ and $w_5$ as $v$ for $T_a$,
\item $w_1$ as $u$ and $w_4$ as $v$ for $T_b$,
\item $w_1$ as $u$ and $w_4$ as $v$ for $T_c$,
\item $w_5$ as $u$ and $w_3$ as $v$ for $T_d$,
\item $w_1$ as $u$ and $w_4$ as $v$ for $T_e$.
\end{itemize}

Consider now $\ell = 5$. Note that for every $T \in \mathcal T$ there is a~copy of $\CC{5}$ with consecutive vertices $w_1, w_2, \ldots, w_5$, and denote it by $C_T$. Each $D_i$ for $0 \leq i < 5$ can be mapped homomorphically into $C_T$ in five different ways. We claim that for every $T \in \mathcal T$ there exists a~homomorphism $Q_5 \to T$ that maps each $D_i$ into $C_T$. Note that if the image of $v_s$ is of out-degree $k$, then for every $0 \leq i < 5$ there are $k$ choices for a~homomorphism $D_i \to C_T$ that agrees with $v_s$, and if the image of $v_t$ is of in-degree $k'$, then there are $k'$ choices for a~homomorphism $D_i \to C_T$ that agrees with $v_t$. Moreover, for every $T \in \{T_a, T_b, T_c, T_d\}$ there exist vertices $u, v \in V(T)$ such that $d^+(u) = 3$, $d^-(v) = 3$, and $uv \in E(T)$. Therefore, if we choose $u$ as the image of $v_s$ and $v$ as the image of $v_t$, then for each $0 \leq i < 5$ there exists a~homomorphism $D_i \to C_T$ agreeing with $v_s$ and $v_t$ simply by the pigeonhole principle. Finally, for $T_e$ it is straightforward to verify that one can choose $w_1$ as the image of $v_s$ and $w_4$ as the image of $v_t$.
\end{proof}

Claims \ref{clm:T-in-6-vertex} and \ref{clm:Q_l-T-3-5} together imply for every $3 \leq \ell \leq 5$ that $Q_\ell$ can be mapped into any strongly connected tournament on $6$ vertices. Hence, to finish the proof of the theorem, it is enough to show that for $\ell \geq 6$ the graph $Q_\ell$ also can be mapped homomorphically into every $T \in \mathcal T$. Note that for every $T \in \mathcal T$, each vertex of $T$ is contained in a~copy of $\CC{3}$. Therefore, if $v_1v_2v_3v_4$ is a directed path in  $D_i$ for some $0 \leq i < \ell$ and neither $v_2$ nor $v_3$ are neighbors of $v_s$ or $v_t$ in $Q_\ell$, we can aim to find a~homomorphism $D_i \to T$ which maps $v_1$ and $v_4$ to the same vertex of $T$, thus essentially reducing the length of $D_i$ by $3$. Since we can always perform this operation as long as the length of the cycle is at least $6$, we can reduce the problem to the case $\ell \leq 5$, which was proved in Claim \ref{clm:Q_l-T-3-5}. \end{proof}

Note that the assumed bound $p(H) \geq 6$ in Theorem~\ref{thm:l-layered} cannot be improved. Indeed, $\CCC{5}{2}$ does not contain two vertices $u$ and $v$ with paths of length $1$, $2$ and $3$ from $u$ to $v$, so the oriented graph $H$ consisting of paths of lengths $1$, $2$, $3$ and $4$ with common endpoints is $\ell$-layered with $p(H)=5$ and $\tau(H) \geq 6$. Analogous constructions can be provided for $p(H)=4$ and $p(H)=3$. In the cases $p(H) \leq 5$ one can easily show that the best bounds are $\tau(H) \leq 2$ when $p(H) = 2$, $\tau(H) \leq 4$ when $p(H) = 3$, and $\tau(H) \leq 6$ when $p(H) \in \{4, 5\}$ and $H$ is $\ell$-layered.

\section{Concluding remarks}

It is straightforward to construct, for any $k > 0$, a~sequence $(H_n)_{n \geq 1}$ of acyclic oriented graphs $H_n \in \Dout{k}$ such that $p(H_n) = n$ and for every $H \in \Dout{k}$ there exists a~homomorphism $H \to H_{p(H)}$. Therefore, to understand the asymptotic behavior of the compressibility of acyclic oriented graphs with out-degree at most $k$, it suffices to examine the sequence $(H_n)_{n \geq 1}$. However, even for $k=2$ we were able to compute $\tau(H_n)$ only for a~few initial values of $n$, and we were unable to find a~lower bound for $\tau(H_n)$ better than linear.

Let $T$ be a~tournament on $11$ vertices $v_0, \ldots, v_{10}$ with arcs $v_iv_{i+j}$ for $j \in \{1, 3, 4, 5, 9\}$ and indices taken modulo $11$. One can verify that every copy of~$\CC{3}$ in $T$ is dominated by some vertex, hence every $H \in \Dout{3}$ can be mapped homomorphically into $T \odot \CC{3}$. Therefore, to prove that $\Dout{3}$ is polynomially $\tau$-bounded, it suffices to show that $T$ satisfies Conjecture \ref{con:EH}. It would be interesting to prove Conjecture \ref{con:EH} for this graph, especially with some low exponent. 

\begin{problem}
For which acyclic oriented graphs $F$ is the family of $F$-free acyclic oriented graphs polynomially $\tau$-bounded?
\end{problem}

Theorem~\ref{thm:2-out-degree_bounded} shows that it holds for $F = \KK{1,3}$. Also, by Proposition \ref{prop:T_3-free_not_z_bounded}, if the family of $F$-free acyclic oriented graphs is polynomially $\tau$-bounded, then $F$ must be bipartite. 

The following definitions and notation are taken from \cite{Sopena16}. We say that an oriented graph~$H$ is an \emph{o-clique} if every two vertices of $H$ are joined by a~directed path of length at most~$2$. Define the \emph{absolute oriented clique number} of $H$, denoted by $\omega_{ao}(H)$, as the maximum size of an o-clique contained in $H$, and the \emph{relative oriented clique number} of $H$, denoted by $\omega_{ro}(H)$, as the maximum size of a~subset $S \subseteq V(H)$ such that every two vertices of $S$ are joined in $H$ by a~directed path of length at most $2$. It is clear that if $H$ is an o-clique and $T$ is a~tournament, then any homomorphism $H \to T$ must be injective, and for a~general oriented graph $H$ we have $\omega_{ao}(H) \leq \omega_{ro}(H) \leq \abs{V(T)}$.
For $k \geq 3$, let $\mathcal A_k$ denote the~family of all acyclic oriented graphs with absolute clique number at most $k$, and let $\mathcal R_k$ denote the~family of all acyclic oriented graphs with relative clique number at most $k$. We have $\mathcal R_k \subseteq \mathcal A_k$ and one may observe that $\Dout{k} \subseteq \mathcal R_{k^2 + 1}$.

\begin{conjecture}
For $k \geq 3$, the families $\mathcal A_k$ and $\mathcal R_k$ are polynomially $\tau$-bounded.
\end{conjecture}

\medskip\noindent
\textbf{Acknowledgments.} We would like to thank Xiaoyu He and Yuval Wigderson for showing that a modification of their proof \cite[Theorem~1.4]{FHW21} provides a quasi-polynomial bound presented in Section~\ref{sec:bounded-degree}.

\end{document}